\documentclass{article}
\usepackage{color,graphicx,amsmath,amsfonts,mathtools,amsthm}
\usepackage[hidelinks]{hyperref}

\theoremstyle{definition}
\newtheorem{definition}{Definition}[section]

\newtheorem{theorem}{Theorem}

\newtheorem{corollary}[theorem]{Corollary}
\newtheorem{proposition}[theorem]{Proposition}
\newtheorem{example}{Example}

\newtheorem{remark}[theorem]{Remark}

\newcommand{\reals}{ \mathbb{R}}
\newcommand{\no}{\mathbb{N}^{+}}

\begin{document}

\title{Ducci Matrices in $p$-adic Context}
\author{Piero Giacomelli   \\  \href{mailto:pgiacome@gmail.com}{pgiacome@gmail.com} }
\date{}
\maketitle

\begin{abstract}
In this paper, we mutuate the concept of Ducci matrices to the $p$-adic setting, generalizing the classical Ducci sequences to the framework of $p$-adic numbers. The classical Ducci operator, which iteratively computes the absolute differences of neighboring elements in a sequence or matrix, is redefined using the $p$-adic absolute value $| \cdot |_p$. We investigate the dynamics of $p$-adic Ducci sequences for matrices over $\mathbb{Q}_p$, focusing on their convergence and periodicity properties.
\end{abstract}

Enrico Ducci an italian mathematician in 1930 \cite{Ciamberlini1937} observed a singular property of the following map

\begin{align*}
    \delta : & \mathbb{N} \longrightarrow \mathbb{N} \\
    (x_1,x_2,x_3,x_4) & \mapsto (|x_1-x_2|, |x_2-x_3|,|x_3-x_4|,|x_4-x_1|)
\end{align*}
iterated with itself always lead to the null t-uple (0,0,0,0). 
The sequence:
\begin{equation*}
 x,\delta(x),\delta^2(x) = \delta(\delta(x))  , \delta^3(x) = \delta(\delta(\delta(x)))  
\end{equation*} 
was defined as the Ducci sequence of $x=(x_1,x_2,x_3,_x4)$. Eighty papers after in 2020 Clausing \cite{clausing2018ducci} defined a wider class of maps of with the Ducci sequence is only a particular case. Meanwhile, in a short note Giacomelli introduced the $p$-adic Ducci operator $D_p$ equivalent of the Ducci operator $D$, but with the valuation norm $|\dot|_p$. In this paper we extend the Ducci matrices in the non-archimedian field $\mathbb{Q}_p$ where $p$ is a prime number. We will start investigating the $p$-adic Ducci matrices to understand the convergence of the starting sequences.
We will start by reviewing some basic definition of the Ducci matrices by recalling the definition.
\begin{definition}
Let $A \in \mathbb{M}_{n\times n}(\mathbb{R})$ a not all zeros squared matrix with real values.    
We define the Ducci operator the following map

\begin{align*}
\delta(x) : & \mathbb{R}^n \longrightarrow  \mathbb{R}^n \\ 
 & x   \longmapsto |Ax|
\end{align*}
when the absolute value is taken componentwise $|x|=(|x_1|,|x_2|,\dots, |x_n|)^\top$.
\end{definition}
It we take $A$ as follows

\begin{align*}
    A =  
\begin{pmatrix}
    1 & -1 & 0 & 0 \\
    0 & 1 & -1 & 0\\
    0 & 0 & 1 & -1\\
    -1 & 0 & 0 & 1
  \end{pmatrix}
\end{align*}

We get the original Ducci map. Starting from a sequence $x \in \reals^n$ the sequence is arranged as 

\begin{align*}
    x, \delta(x), \delta^2= \delta(\delta(x))= \delta \circ\delta, \dots, \delta^n(x), \dots  
\end{align*}
will be called the Ducci sequence of $x$ respect to $A$. It $x$ is a zero vector (i.e all the entries are zeros), we say that the Ducci sequence terminates. Moreover:

\begin{definition}
We say that the Ducci map is cyclic if there exist and index $n$ a $k \in \no$ and a not null vector $x$ such that 
\begin{equation*}
    \delta^n(x) = \delta^{n+k}(x).
\end{equation*}
$k$ is called the length of the cycle.
\end{definition}

And in a natural way the matrix $A \in \mathbb{M}_{n\times n}(\mathbb{R})$ is called the Ducci matrix associated to the Ducci map.

Some questions arise from the previous definition:
\begin{itemize}
    \item Considering a Ducci matrix $A$ and non zero starting sequence $x \in \mathbb{R}$ does the Ducci sequence associated terminate?
    \item Considering a Ducci matrix $A$ what are the sequences that enter a cycle and in case what are the cycle of minumum lenght?
    \item Considering a Ducci matrix $A$ and the set of all starting sequences that terminates are the fastest (ie become null in the minimum number of iteration)?
\end{itemize}
In this paper we are interested to extend the previous definition in the using the $p$-adic valuation instead of the standard absolute value norm and we will try to answer some of the previous questions in the $p$-adic field $\mathbb{Q}_p$.
Using the previous notation we can move with the following
\begin{definition}
    Let $p$ be a prime and let $x$ be an n-tuple with entries in $\mathbb{Q}_p$ so that $x \in \mathbb{Q}_p^{n}$ and $D_p \in \mathbb{M}_{n\times n}(\mathbb{Q}_p)$ a matrix whose entries are in the field  $\mathbb{Q}_p$ we define the $p$-adic Ducci operator $\delta_p(x)$ as follows
    \begin{align*}
\delta_p(x) : & \mathbb{Q}_p^n \longrightarrow  \mathbb{Q}_p^n \\ 
 & x   \longmapsto \delta_p(x)= |D_p x|_p
\end{align*}
where if $x \in \mathbb{Q}^n_p$ then $|x|_p = (|x_1|_p,|x_2|_p,\dots,|x_n|_p)$ and for every $i$ we have that $|x_i|_p = \frac{1}{p^{ord_p(x_i)}}$ being $ord_p(x_i) = max\{m: p^m | x_i\}$ (i.e. $ord_p(x)$ is the maximum power of $p$ that divide $x_i$).
\end{definition}
Prior to proceed let us consider the the product \( D_p x \). Since \( D_p \) is an \( n \times n \) matrix with entries in \( \mathbb{Q}_p \) and \( x \) is an \( n \)-tuple with entries in \( \mathbb{Q}_p \), the result \( D_p x \) is an \( n \)-tuple with entries in \( \mathbb{Q}_p \). Formally, if \( D_p = (d_{ij}) \) and \( x = (x_1, x_2, \ldots, x_n) \), then the \( i \)-th component of \( D_p x \) is given by:
\[
(D_p x)_i = \sum_{j=1}^n d_{ij} x_j \in \mathbb{Q}_p
\]

The p-adic norm \( |\cdot|_p \) is applied component-wise to the resulting \( n \)-tuple \( D_p x \). For each component \( (D_p x)_i \) of \( D_p x \), we have:
\[
|(D_p x)_i|_p = \frac{1}{p^{\text{ord}_p((D_p x)_i)}}
\]
where \( \text{ord}_p((D_p x)_i) \) is the highest power of \( p \) dividing \( (D_p x)_i \).

Since the p-adic norm \( |\cdot|_p \) maps elements of \( \mathbb{Q}_p \) to \( \mathbb{Q}_p \), applying it component-wise to an \( n \)-tuple will result in another \( n \)-tuple with entries in \( \mathbb{Q}_p \). Therefore, \( |D_p x|_p \) is an \( n \)-tuple with entries in \( \mathbb{Q}_p \).

Thus, \( \delta_p(x) = |D_p x|_p \in \mathbb{Q}_p^n \), hence, the map \( \delta_p \) is well-defined.

As we have seens previous we can define the $p$-adic Ducci sequence the sequence
\begin{definition}[$p$-adic Ducci sequence]
\begin{align*}
    x, \delta_p(x), \delta_p^2= \delta_p(\delta_p(x))= \delta_p \circ\delta_p, \dots, \delta_p^n(x), \dots  
\end{align*}    
\end{definition}
with initial seed $x \in \mathbb{Q}_p^n$ respect to $D_p \in \mathbb{M}_{n\times n}(\mathbb{Q}_p)$.
In this p-adic context, we define a p-adic Ducci map associated with a $p$-adic matrix $D_p$ as a function that maps a p-adic vector $x$ to $|D_px|_p$, where $| \cdot |_p$ is applied elementwise. 

Without loss of generality we can define the following 
For convergence to zero, we analyze when $x_k \to 0$ as $k \to \infty$.

Prior to move forward we recall  the following

\begin{remark}
In the field of \( p \)-adic numbers \( \mathbb{Q}_p \), the \( p \)-adic absolute value \( |x|_p \) of a nonzero element \( x \) is defined as:

\[
|x|_p = p^{-v_p(x)},
\]

where \( v_p(x) \) is the \( p \)-adic valuation of \( x \). The valuation \( v_p(x) \) is the highest power of \( p \) that divides \( x \) in \( \mathbb{Q}_p \).

If \( |x|_p < 1 \) and $x \neq 0$, this means that \( v_p(x) > 0 \), so \( x \) is divisible by \( p \). The possible values of \( |x|_p \) in this case are:

\[
|x|_p = p^{-k}, \quad \text{where } k \in \mathbb{N} \text{ and } k \geq 1.
\]

Thus, the possible values of \( |x|_p \) when \( |x|_p < 1 \) are:

\[
\left\{ p^{-1}, p^{-2}, p^{-3}, \dots \right\}.
\]
So in $|x|_p<1 \iff x \in {0} \cup \{p^{-1},p^{-2},\dots p^{k},\dots\} $.
The condition \( |x|_p = 1 \) implies:
\[
p^{-v_p(x)} = 1.
\]
Since \( p^{-v_p(x)} = 1 \) if and only if \( v_p(x) = 0 \), the condition \( |x|_p = 1 \) is equivalent to \( v_p(x) = 0 \).
So , if \( v_p(x) = 0 \), then \( x \) is a \( p \)-adic integer (i.e., \( x \in \mathbb{Z}_p \)) and \( x \) is not divisible by \( p \).
\end{remark}

We are now ready to answer some questions about the $p$-adic Ducci sequence. The first one is find a necessary condition such that given any initial not null sequence $x$, eventually the $p$-adic Ducci sequence terminates.

Let us consider the first simpler case in the following
\begin{proposition}
If all eigenvalues of $D_p \in \mathbb{M}_{n\times n} (\mathbb{Q}_p)$  have $p$-adic norm less that $1$ then the $p$-adic Ducci sequence terminates.       
\end{proposition}
\begin{proof}
Let $D_p$ be a $p$-adic diagonal matrix:
\begin{equation*}
    A = \text{diag}(\lambda_1, \lambda_2, \dots, \lambda_n),
\end{equation*}
where each $\lambda_i$ satisfies $|\lambda_i|_p < 1$. Given an initial vector $x_0 = (x_1, x_2, \dots, x_n)^\top \in \mathbb{Q}_p^n$, the iterates satisfy:
\begin{equation*}
    x_{k+1} = D_p^k x_k.
\end{equation*}
Since $D_p$ is diagonal, each component $x_i$ of the $n$-tuple $x$ follows:
\begin{equation*}
    x_i^{k+1} = \lambda_i^k x_i^{k},
\end{equation*}
with $0\leq i \leq n$ and $k \in\mathbb{N}$
Applying the $p$-adic absolute value:
\begin{equation*}
    |x_k^{(i)}|_p = |\lambda_i^k x_0^{(i)}|_p = |\lambda_i|_p^k \cdot |x_0^{(i)}|_p.
\end{equation*}
Since $|\lambda_i|_p < 1$, we have:
\begin{equation*}
    \lim_{k \to \infty} |\lambda_i|_p^k = 0.
\end{equation*}
Thus,
\begin{equation*}
    \lim_{k \to \infty} |x_k^{(i)}|_p = 0.
\end{equation*}
Since this holds for each coordinate $i$, it follows that:
\begin{equation*}
    \lim_{k \to \infty} |(D_p)^k x_0|_p = 0.
\end{equation*}
This proves that the sequence eventually terminates becoming  the zero vector in $\mathbb{Q}_p^n$.
\end{proof}
We can also have the following corollary
\begin{corollary}
Let $D_p$ a $p$-adic Ducci matrix if all eigenvalues are in the set $\{p^{-1},\dots,p^{-n},\dots\}$ then the $p$-adic Ducci sequence terminates. 
\end{corollary}

A more interesting case of study for the $p$-adic Ducci matrix is the following one

\begin{theorem}
Let \( D_p \in \mathbb{M}_{n \times n}(\mathbb{Q}_p) \) be a \( p \)-adic Ducci matrix, and let \( x_0 \in \mathbb{Q}_p^n \) be an initial vector. If all eigenvalues \( \lambda_i \) of \( D_p \) satisfy \( |\lambda_i|_p = 1 \), then:
\begin{enumerate}
    \item The \( p \)-adic Ducci sequence \( \{x_k\} \) defined by \( x_{k+1} = |D_p x_k|_p \) does not converge to zero.
    \item If the eigenvalues are roots of unity in \( \mathbb{Q}_p \), the sequence \( \{x_k\} \) is eventually periodic.
\end{enumerate}
\end{theorem}

\begin{proof}

Let us start with the first part about the non-convergence to zero.
\newline
Let \( \lambda_1, \lambda_2, \dots, \lambda_n \) be the eigenvalues of \( D_p \), and assume \( |\lambda_i|_p = 1 \) for all \( i \). By the Jordan canonical form, \( D_p \) can be decomposed as:
\[
D_p = P J P^{-1},
\]
where \( P \) is an invertible matrix and \( J \) is the Jordan canonical form of \( D_p \). The Jordan blocks of \( J \) correspond to the eigenvalues \( \lambda_i \).

For each eigenvalue \( \lambda_i \), the \( p \)-adic norm satisfies \( |\lambda_i|_p = 1 \). Consider the iterates of the sequence:
\[
x_{k+1} = |D_p x_k|_p.
\]
In the Jordan basis, the iterates can be expressed as:
\[
y_{k+1} = |J y_k|_p,
\]
where \( y_k = P^{-1} x_k \). Since \( J \) is block-diagonal, the behavior of \( y_k \) is determined by the Jordan blocks of \( J \).

For each Jordan block corresponding to \( \lambda_i \), the iterates satisfy:
\[
y_k^{(i)} = \lambda_i^k y_0^{(i)} + \text{(lower-order terms)},
\]
where the lower-order terms arise from the structure of the Jordan block. Applying the \( p \)-adic norm:
\[
|y_k^{(i)}|_p = |\lambda_i^k y_0^{(i)} + \text{(lower-order terms)}|_p.
\]
Since \( |\lambda_i|_p = 1 \), we have:
\[
|\lambda_i^k y_0^{(i)}|_p = |\lambda_i|_p^k \cdot |y_0^{(i)}|_p = |y_0^{(i)}|_p.
\]
The lower-order terms do not affect the \( p \)-adic norm because \( |\lambda_i|_p = 1 \) ensures that the dominant term is \( \lambda_i^k y_0^{(i)} \). Thus:
\[
|y_k^{(i)}|_p = |y_0^{(i)}|_p.
\]
This shows that the \( p \)-adic norm of each component of \( y_k \) remains constant. Consequently, the sequence \( \{x_k\} \) does not converge to zero.
\newline
We would like to move on by examinating the periodicity of the roots of unity so to determinate the periodicity in this case.
Assume that the eigenvalues \( \lambda_i \) are roots of unity in \( \mathbb{Q}_p \). That is, for each \( \lambda_i \), there exists an integer \( m_i \geq 1 \) such that:
\[
\lambda_i^{m_i} = 1.
\]
Let \( m \) be the least common multiple of the \( m_i \). Then, for each \( \lambda_i \), we have:
\[
\lambda_i^m = 1.
\]
In the Jordan basis, the iterates satisfy:
\[
y_{k+m}^{(i)} = \lambda_i^{k+m} y_0^{(i)} + \text{(lower-order terms)} = \lambda_i^k y_0^{(i)} + \text{(lower-order terms)} = y_k^{(i)}.
\]
Thus, the sequence \( \{y_k\} \) is periodic with period \( m \). Transforming back to the original basis, the sequence \( \{x_k\} \) is also periodic with period \( m \).

\end{proof}

\begin{theorem}
Let \( D_p \in \mathbb{M}_{n \times n}(\mathbb{Q}_p) \) be a \( p \)-adic Ducci matrix, and let \( x_0 \in \mathbb{Z}_p^n \) be an initial vector. If all eigenvalues \( \lambda_i \) of \( D_p \) satisfy \( |\lambda_i|_p = 1 \), then:
\begin{enumerate}
    \item The \( p \)-adic Ducci sequence \( \{x_k\} \) defined by \( x_{k+1} = |D_p x_k|_p \) does not converge to zero.
    \item If the eigenvalues are roots of unity in \( \mathbb{Q}_p \), the sequence \( \{x_k\} \) is eventually periodic.
\end{enumerate}
\end{theorem}

\begin{proof}

\noindent \textbf{Part 1: Non-convergence to zero}

Let \( \lambda_1, \lambda_2, \dots, \lambda_n \) be the eigenvalues of \( D_p \), and assume \( |\lambda_i|_p = 1 \) for all \( i \). By the Jordan canonical form, \( D_p \) can be decomposed as:
\[
D_p = P J P^{-1},
\]
where \( P \) is an invertible matrix and \( J \) is the Jordan canonical form of \( D_p \). The Jordan blocks of \( J \) correspond to the eigenvalues \( \lambda_i \).

For each eigenvalue \( \lambda_i \), the \( p \)-adic norm satisfies \( |\lambda_i|_p = 1 \). Consider the iterates of the sequence:
\[
x_{k+1} = |D_p x_k|_p.
\]
In the Jordan basis, the iterates can be expressed as:
\[
y_{k+1} = |J y_k|_p,
\]
where \( y_k = P^{-1} x_k \). Since \( J \) is block-diagonal, the behavior of \( y_k \) is determined by the Jordan blocks of \( J \).

For each Jordan block corresponding to \( \lambda_i \), the iterates satisfy:
\[
y_k^{(i)} = \lambda_i^k y_0^{(i)} + \text{(lower-order terms)},
\]
where the lower-order terms arise from the structure of the Jordan block. Applying the \( p \)-adic norm:
\[
|y_k^{(i)}|_p = |\lambda_i^k y_0^{(i)} + \text{(lower-order terms)}|_p.
\]
Since \( |\lambda_i|_p = 1 \), we have:
\[
|\lambda_i^k y_0^{(i)}|_p = |\lambda_i|_p^k \cdot |y_0^{(i)}|_p = |y_0^{(i)}|_p.
\]
The lower-order terms do not affect the \( p \)-adic norm because \( |\lambda_i|_p = 1 \) ensures that the dominant term is \( \lambda_i^k y_0^{(i)} \). Thus:
\[
|y_k^{(i)}|_p = |y_0^{(i)}|_p.
\]
This shows that the \( p \)-adic norm of each component of \( y_k \) remains constant. Consequently, the sequence \( \{x_k\} \) does not converge to zero.

Moving with the second part of the theorem we now assume that the eigenvalues \( \lambda_i \) are roots of unity in \( \mathbb{Q}_p \). That is, for each \( \lambda_i \), there exists an integer \( m_i \geq 1 \) such that:
\[
\lambda_i^{m_i} = 1.
\]
Let \( m \) be the least common multiple of the \( m_i \). Then, for each \( \lambda_i \), we have:
\[
\lambda_i^m = 1.
\]
In the Jordan basis, the iterates satisfy:
\[
y_{k+m}^{(i)} = \lambda_i^{k+m} y_0^{(i)} + \text{(lower-order terms)} = \lambda_i^k y_0^{(i)} + \text{(lower-order terms)} = y_k^{(i)}.
\]
Thus, the sequence \( \{y_k\} \) is periodic with period \( m \). Transforming back to the original basis, the sequence \( \{x_k\} \) is also periodic with period \( m \).

\end{proof}

When the eigenvalues of the $p$-adic Ducci matrix $D_p$ are not bounded is it possible to find a sequence that does not terminates and does not enter a cycle. 

\begin{example}

Let \( p = 2 \) (for simplicity), and consider the following \( 2 \times 2 \) matrix over \( \mathbb{Q}_2 \):

\[
D_2 = \begin{pmatrix}
\frac{1}{2} & 0 \\
0 & \frac{1}{2}
\end{pmatrix}.
\]

The eigenvalues of \( D_2 \) are \( \lambda_1 = \frac{1}{2} \) and \( \lambda_2 = \frac{1}{2} \). The \( 2 \)-adic norm of these eigenvalues is:
\[
|\lambda_1|_2 = |\lambda_2|_2 = \left|\frac{1}{2}\right|_2 = 2 > 1.
\]

For the matrix \( D_2 \), the iterates are:
\[
x_{k+1} = \left| \begin{pmatrix}
\frac{1}{2} & 0 \\
0 & \frac{1}{2}
\end{pmatrix} x_k \right|_2 = \left( \left|\frac{1}{2} x_1^{(k)}\right|_2, \left|\frac{1}{2} x_2^{(k)}\right|_2 \right)^\top.
\]

Using the properties of the \( p \)-adic norm:
\[
\left|\frac{1}{2} x_i^{(k)}\right|_2 = \left|\frac{1}{2}\right|_2 \cdot \left|x_i^{(k)}\right|_2 = 2 \cdot \left|x_i^{(k)}\right|_2.
\]

Thus, the iteration becomes:
\[
x_{k+1} = 2 \cdot x_k.
\]

Starting from \( x_0 = (x_1, x_2)^\top \), the sequence grows exponentially:
\[
x_1 = 2 x_0, \quad x_2 = 2 x_1 = 2^2 x_0, \quad x_3 = 2 x_2 = 2^3 x_0, \quad \dots, \quad x_k = 2^k x_0.
\]

The \( p \)-adic norm of \( x_k \) is:
\[
|x_k|_2 = |2^k x_0|_2 = |2^k|_2 \cdot |x_0|_2 = 2^{-k} \cdot |x_0|_2.
\]

Since \( |x_0|_2 \) is fixed and \( 2^{-k} \) grows without bound as \( k \to \infty \), the sequence \( \{x_k\} \) grows indefinitely in the \( p \)-adic norm.

The matrix \( D_2 = \begin{pmatrix} \frac{1}{2} & 0 \\ 0 & \frac{1}{2} \end{pmatrix} \) has eigenvalues \( \lambda_1 = \lambda_2 = \frac{1}{2} \), which satisfy \( |\lambda_i|_2 = 2 > 1 \). The associated \( p \)-adic Ducci sequence grows exponentially and diverges in the \( p \)-adic norm. This provides an explicit example of a \( p \)-adic Ducci matrix with eigenvalues greater than 1 in \( p \)-adic valuation that leads to indefinite growth.
\end{example}

The fact that the Ducci matrix has values in $\mathbb{M}_{n \times n}(\mathbb{Q}_p)$ and the initial sequence $x$ has values in $\mathbb{Q}_p^n$ is essential for the previous behavior, a simple change in this assumption can lead to a different result as stated by the following:

\begin{theorem}
Let \( D_p \in \mathbb{M}_{n \times n}(\mathbb{Z}_p) \) be a \( p \)-adic Ducci matrix, and let \( x_0 \in \mathbb{Z}_p^n \) be an initial vector. If all eigenvalues \( \lambda_i \) of \( D_p \) satisfy \( |\lambda_i|_p = 1 \), then:
\begin{enumerate}
    \item The \( p \)-adic Ducci sequence does not terminate.
    \item If the eigenvalues are roots of unity in \( \mathbb{Q}_p \), the sequence \( \{x_k\} \) is eventually periodic.
\end{enumerate}
\end{theorem}

\begin{proof}

Let \( \lambda_1, \lambda_2, \dots, \lambda_n \) be the eigenvalues of \( D_p \), and assume \( |\lambda_i|_p = 1 \) for all \( i \). By the Jordan canonical form, \( D_p \) can be decomposed as:
\[
D_p = P J P^{-1},
\]
where \( P \) is an invertible matrix and \( J \) is the Jordan canonical form of \( D_p \). The Jordan blocks of \( J \) correspond to the eigenvalues \( \lambda_i \).

For each eigenvalue \( \lambda_i \), the \( p \)-adic norm satisfies \( |\lambda_i|_p = 1 \). Consider the iterates of the sequence:
\[
x_{k+1} = |D_p x_k|_p.
\]
In the Jordan basis, the iterates can be expressed as:
\[
y_{k+1} = |J y_k|_p,
\]
where \( y_k = P^{-1} x_k \). Since \( J \) is block-diagonal, the behavior of \( y_k \) is determined by the Jordan blocks of \( J \).

For each Jordan block corresponding to \( \lambda_i \), the iterates satisfy:
\[
y_k^{(i)} = \lambda_i^k y_0^{(i)} + \text{(lower-order terms)},
\]
where the lower-order terms arise from the structure of the Jordan block. Applying the \( p \)-adic norm:
\[
|y_k^{(i)}|_p = |\lambda_i^k y_0^{(i)} + \text{(lower-order terms)}|_p.
\]
Since \( |\lambda_i|_p = 1 \), we have:
\[
|\lambda_i^k y_0^{(i)}|_p = |\lambda_i|_p^k \cdot |y_0^{(i)}|_p = |y_0^{(i)}|_p.
\]
The lower-order terms do not affect the \( p \)-adic norm because \( |\lambda_i|_p = 1 \) ensures that the dominant term is \( \lambda_i^k y_0^{(i)} \). Thus:
\[
|y_k^{(i)}|_p = |y_0^{(i)}|_p.
\]
This shows that the \( p \)-adic norm of each component of \( y_k \) remains constant. Consequently, the sequence \( \{x_k\} \) does not converge to zero.

For the second part we can reuse the previous results.
Assume that the eigenvalues \( \lambda_i \) are roots of unity in \( \mathbb{Q}_p \). That is, for each \( \lambda_i \), there exists an integer \( m_i \geq 1 \) such that:
\[
\lambda_i^{m_i} = 1.
\]
Let \( m \) be the least common multiple of the \( m_i \). Then, for each \( \lambda_i \), we have:
\[
\lambda_i^m = 1.
\]
In the Jordan basis, the iterates satisfy:
\[
y_{k+m}^{(i)} = \lambda_i^{k+m} y_0^{(i)} + \text{(lower-order terms)} = \lambda_i^k y_0^{(i)} + \text{(lower-order terms)} = y_k^{(i)}.
\]
Thus, the sequence \( \{y_k\} \) is periodic with period \( m \). Transforming back to the original basis, the sequence \( \{x_k\} \) is also periodic with period \( m \).

\end{proof}

In this paper we present the "dopplerganger" of the Ducci matrices and the Ducci operatorin the context of the $p$-adic field $\mathbb{Q}_p$ and $\mathbb{Z}_p$ both as matrices as well as starting sequences. We studied different cases and we enlight some convergence criteria such that the original sequence terminates. We did not examine the velocity (ie. the mininum not null sequence that terminates in the fastest way. This will be object of further research.

\bibliographystyle{ieeetr}
\bibliography{references}

\end{document}